\documentclass[12pt]{article}

\usepackage{mathtools,amssymb,amsthm,amsmath,amsfonts} 
          
\newtheorem{theorem}{Theorem}
\newtheorem{lemma}{Lemma}
\newcommand{\ux}{\underline{x}}
\newcommand{\uy}{\underline{y}}
\newcommand{\bl}{{\bf l}}

\newcommand{\bj}{{\bf j}}

\newcommand{\bC}{{\bf C}}

\newcommand{\C}{\mathbb{C}}
\newcommand{\poz}{\partial_{\overline{z}}}
\newcommand{\R}{\mathbb{R}}

\newcommand{\Li}{{\mbox{Lip}}}
\newcommand{\cS}{\mathcal{S}}
\newcommand{\cP}{\mathcal{P}}

\renewcommand{\iint}{\int_{\Gamma}\int_{\Gamma}}
\newcommand{\dd}{d\tau d\zeta}
\title{Polyanalytic Hardy decomposition of higher order Lipschitz functions}
\author{Ricardo Abreu Blaya$^{1}$, Lianet De la Cruz Toranzo$^{2}$}
\date {\small{{$^1${Facultad de Matem\'aticas, Universidad Aut\'onoma de Guerrero, M\'exico.}\\e-mail: rabreublaya@yahoo.es\\$^2${Facultad de Inform\'atica y Matem\'atica, Universidad de Holgu\'in, Cuba.}\\e-mail: lcruzt@uho.edu.cu}}\\\vspace{.5cm}{\textit{Dedicated to Professor Juan Bory Reyes, the teacher of us both}}}
\begin{document}
\maketitle
\begin{abstract}\noindent
This paper is concerned with the problem of decomposing a higher order Lipschitz function on a closed Jordan curve $\Gamma$ into a sum of two polyanalytic functions in each open domain defined by $\Gamma$. Our basic tools are the Hardy projections related to a singular integral operator arising in polyanalytic function theory, which, as it is proved here, represents an involution operator on the higher order Lipschitz classes. Our result generalizes the classical Hardy decomposition of H\"older continuous functions on the boundary of a domain.
\end{abstract}
\small{
\noindent
\textbf{Keywords.} Polyanalytic functions, higher order Lipschitz classes, Hardy projections.\\
\noindent
\textbf{AMS Classification (2000).} 30G35.}
\section{Introduction}
Decomposing a complex function on a closed Jordan curve $\Gamma$ as the sum of the traces of two holomorphic functions, one in each open domain defined by the curve, is a classical topic from complex analysis and represents a cornerstone in the solutions of Riemann-Hilbert problems. Indeed, such a decomposition features the Hilbert transform (a singular version of the Cauchy transform) whose $\bC^{0,\alpha}$ and $L_p$ boundedness plays a fundamental role in real harmonic analysis \cite{Da,Ga,Mu,StZy}.  \\
\noindent
A higher order Lipschitz data on $\Gamma$ (in the sense of Whitney) is the collection of real-valued continuous functions
\begin{equation}\label{LD}
f:=\{f^{(\bj)},\,|\bj|\le k\}
\end{equation} 
defined on $\Gamma$ and satisfying the compatibility conditions
\begin{equation}\label{C}
|f^{(\bj)}(\ux)-\sum_{|\bj+\bl|\le k}\frac{f^{(\bj+\bl)}(\uy)}{\bl !}(\ux-\uy)^\bl|=\mathcal{O}(|\ux-\uy|^{k+\alpha-|\bj|}),\,\,\ux,\uy\in\Gamma,\,|\bj|\le k.
\end{equation}
Here we use the classical multi-index notation  
\[
\ux^\bl := x_1^{l_1}x_2^{l_2},\,\bj!:=j_1!j_2!,\,|\bj|:=j_1+j_2,\,\bj+\bl=(j_1+l_1,j_2+l_2),
\]
where $\bj= (j_1,j_2)$ and $\bl= (l_1,l_2)$ are $2$-dimensional multi-indices in ${\mathbb{N}}_0^2$.\\
\noindent
A deep theorem from real analysis due to Whitney \cite{Wh}, shows that such a data may be extended to $\R^2$ in a $\bC^{k,\alpha}$-smooth way.\\ 
\noindent
For $k$ a non-negative integer and $0<\alpha\le 1$, the higher order Lipschitz class $\Li(k+\alpha,\Gamma)$ consists of data \eqref{LD} satisfying \eqref{C}. We shall say that a complex data $f=u+iv$ belongs to $\Li(k+\alpha,\Gamma)$  if both $u$ and $v$ do so. In this context, however, this definition can be reformulated in purely complex terms by appealing to the complexified compatibility conditions  
\begin{equation}\label{C2}
|f^{(\bj)}(t)-\sum\limits_{|\bj+\bl|\le k}\frac{f^{(\bj+\bl)}(\tau)}{\bl!}(t-\tau)^{l_1}(\overline{t-\tau})^{l_2}|=\mathcal{O}(|t-\tau|^{k+\alpha-|\bj|}),\,\,t,\tau \in\Gamma,\,|\bj|\le k.
\end{equation}
A useful fact to note is that if $f\in\Li(k+\alpha,\Gamma)$ and we set $\bj=(j_1,j_2)$, with $|\bj|\le k$, then  
\begin{equation}\label{subdata}
f_{(\bj)}:=\{f^{(\bj+\bl)},\,0\le|\bl|\le k-|\bj|\}
\end{equation}  
belongs to $\Li(k+\alpha-|\bj|,\Gamma)$.\\
\noindent
For further use, let us observe that while the superscript $(\bj)$ in $f$ denotes a single element, the subscript denotes rather a sub-collection of $f$. Moreover, the following duality relation holds: 
\begin{equation}\label{dual}
[f_{(\bj)}]^{(\bl)}=f^{(\bj+\bl)}, |\bl|\leq k-|\bj|.
\end{equation}
Recently, a version of the Plemelj-Privalov theorem on the invariance of $\Li(k+\alpha, \Gamma)$, with respect to the principal value singular integral operator 
\begin{equation*}\label{SIO}
[\cS_k f]^{(0,0)}(t):=\sum_{n=0}^{k}\frac{1}{\pi i}\int_{\Gamma}\frac{(\overline{t-\zeta})^{n}}{n!(\zeta - t)}f^{(0,n)}(\zeta)d\zeta,\;t\in\Gamma,
\end{equation*}
related to polyanalytic function theory has been proved in \cite{DAB} (see also \cite{BDA}). This result must be interpreted in the sense that, for $f\in\Li(k+\alpha,\Gamma)$, the collection
\[
\cS_k f=\bigg\{[\cS_k f]^{(\bj)},\,|\bj|\le k\bigg\}:=\bigg\{\dfrac{\bj!}{\pi i}\int_{\Gamma}\dfrac{\sum_{p=j_2}^{k}\binom{p}{j_2}(\overline{t-\zeta})^{p-j_2}{f^{(0,p)}(\zeta)}/{p!}}{(\zeta-t)^{j_1+1}}d\zeta,\,|\bj|\le k\bigg\}
\]
is indeed in $\Li(k+\alpha,\Gamma)$. Here and throughout, $\Gamma$ denotes the boundary of a domain $\Omega\subset\C$ with interior $\Omega_+$ and exterior $\Omega_-.$\\
\noindent
The first aim of this paper is to prove that $\cS_k$ behaves as an involution operator on $\Li(k+\alpha,\Gamma)$, and then take advantage of this fact to define the Plemelj projection operators
\[
\cP_k^+=\frac{1}{2}(I+\cS_k),\,\,\cP_k^-=\frac{1}{2}(I-\cS_k)
\]
giving rise to the Hardy decomposition of $\Li(k+\alpha, \Gamma)$ by polyanalytic functions (of order $k+1$).
\section{Involution property}
As it is well known, a linear operator is an involution if and only if it equals its own inverse. However, a first look at \eqref{SIO} might suggest that the operator $\cS_k$ is not even an injection. Indeed, it is defined only by some elements of the Lipschitz data $f:=\{f^{(\bj)},\,|\bj|\le k\}$. The explanation for this possible misunderstanding lies in the fact that even though a Lipschitz data is not determined by none of its proper sub-collections (see \cite[p. 176]{St}), it can easily be proved that, in the case of a curve or any compact set without isolated points, the elements $\{f^{(0,n)},\,0\le n\le k\}$ actually determine the Lipschitz data $f$.

\noindent
We prove this fact by induction on $k$. Let $k=1$ and suppose $\{f^{(0,0)},f^{(1,0)},f^{(0,1)}\}$ and $\{f^{(0,0)},f_*^{(1,0)},f^{(0,1)}\}$ are two different Lipschitz data in $\Li(1+\alpha,\Gamma)$. Then
\begin{eqnarray*}
(t-\tau)[f_*^{(1,0)}(\tau)-f^{(1,0)}(\tau)]=\\\big[f^{(0,0)}(t)-f^{(0,0)}(\tau)-(t-\tau)f^{(1,0)}(\tau)-(\overline{t-\tau})f^{(0,1)}(\tau)\big]-\\\big[f^{(0,0)}(t)-f^{(0,0)}(\tau)-(t-\tau)f_*^{(1,0)}(\tau)-(\overline{t-\tau})f^{(0,1)}(\tau)\big].
\end{eqnarray*}
By setting $\bj=(0,0)$ in \eqref{C2}, we have
\[
|(t-\tau)(f_*^{(1,0)}(\tau)-f^{(1,0)}(\tau))|=\mathcal{O}(|t-\tau|^{1+\alpha}),\,\,t,\tau\in\Gamma
\]
and hence
\[
|f_*^{(1,0)}(\tau)-f^{(1,0)}(\tau)|=\mathcal{O}(|t-\tau|^{\alpha}),\,\,t,\tau\in\Gamma.
\]
Of course, the last condition forces the identity $f_*^{(1,0)}(\tau)=f^{(1,0)}(\tau),\,\tau\in\Gamma$.\\
Assuming the property to hold for $n=k-1$, we will prove it for $n=k$. Given two Lipschitz data $f:=\{f^{(\bj)},\,|\bj|\le k\}$ and $g:=\{g^{(\bj)},\,|\bj|\le k\}$ satisfying $f^{(0,j_2)}=g^{(0,j_2)}$ for $0\leq j_2\leq k$, we shall prove that actually $f^{(\bj)}=g^{(\bj)}$ for all $|\bj|\leq k$. \\
Let us consider the sub-collections 
\[
f_{(0,1)}=\{f^{(j_1,j_2+1)},\,|\bj|\le k-1\}\,\,\mbox{and}\,\, g_{(0,1)}=\{g^{(j_1,j_2+1)},\,|\bj|\le k-1\}. 
\]
We recall that both $f_{(0,1)}$ and $g_{(0,1)}$ belong to $\Li(k-1+\alpha,\Gamma)$.\\ 
Since $[f_{(0,1)}]^{(0,j_2)}=f^{(0,j_2+1)}$ and $f^{(0,j_2+1)}=g^{(0,j_2+1)}$ for $0\leq j_2+1\leq k$ we have $[f_{(0,1)}]^{(0,j_2)}=[g_{(0,1)}]^{(0,j_2)}$ for $0\leq j_2\leq k-1$. Then, 
by the induction hypothesis, we conclude that $[f_{(0,1)}]^{(\bj)}=[g_{(0,1)}]^{(\bj)}$ for $0\leq |\bj|\leq k-1$, that is, $f^{(j_1,j_2+1)}=g^{(j_1,j_2+1)}$ for $|\bj|\le k-1$, or equivalently $f^{(l_1,l_2)}=g^{(l_1,l_2)}\,\,\mbox{for} \,\,|\bl|\le k,\,l_2\ge 1$. Consequently, we are reduced to proving $f^{(l_1,0)}=g^{(l_1,0)}\,\,\mbox{for} \,\,1\le l_1\le k$.\\
It is easily seen that
\begin{eqnarray}\label{truco}
\sum\limits_{l_1=1}^k\bigg[\frac{g^{(l_1,0)}(\tau)-f^{(l_1,0)}(\tau)}{l_1!}\bigg](t-\tau)^{l_1}=\nonumber\\\bigg[f^{(0,0)}(t)-\sum\limits_{|\bl|\le k}\frac{f^{(\bl)}(\tau)}{\bl!}(t-\tau)^{l_1}(\overline{t-\tau})^{l_2}\bigg]\nonumber\\
-\bigg[g^{(0,0)}(t)-\sum\limits_{|\bl|\le k}\frac{g^{(\bl)}(\tau)}{\bl!}(t-\tau)^{l_1}(\overline{t-\tau})^{l_2}\bigg].
\end{eqnarray}
By setting $\bj=(0,0)$ in \eqref{C2}, we have
\[
\bigg|\sum\limits_{l_1=1}^k\bigg[\frac{g^{(l_1,0)}(\tau)-f^{(l_1,0)}(\tau)}{l_1!}\bigg](t-\tau)^{l_1-1}\bigg|=\mathcal{O}(|t-\tau|^{k-1+\alpha}),\,\,t,\tau\in\Gamma
\]
and hence that $g^{(1,0)}(\tau)=f^{(1,0)}(\tau),\,\tau\in\Gamma$. Our claim follows after one replaces this last equality in \eqref{truco} and repeats this procedure enough times.

\noindent
The proof of the following auxiliary result is straightforward.
\begin{lemma}\label{lema1}
If $t,\tau\in\Gamma$, then $\displaystyle\dfrac{1}{\pi i}\int_{\Gamma}\dfrac{d\zeta}{(\zeta-t)^m(\zeta-\tau)^n}=0$ for all natural numbers $m,n\geq 1$.
\end{lemma}
\noindent
We can now state and prove the main result of this section.
\begin{theorem}
The operator $\cS_k:\emph{\Li}(k+\alpha,\Gamma)\mapsto \emph{\Li}(k+\alpha,\Gamma)$ is a linear involution. That is,
\begin{equation}\label{invo}
[\cS_k^2f]^{(\bj)}=f^{(\bj)}\,\,\mbox{for all}\,\, |\bj|\leq k.
\end{equation}
\end{theorem}
\begin{proof}
Since $\cS_k^2f\in\Li(k+\alpha,\Gamma)$ is completely determined by its $(0,j_2)$-th components, we only need to show that $[\cS_k^2f]^{(0,j_2)}=f^{(0,j_2)}$ for $0\le j_2\leq k$. The proof is by induction on $k$. \\
The base step $k=0$ is well known in the literature \cite{Ga}. Assuming 
\begin{equation}\label{invo2}
[\cS_n^2f]^{(0,j_2)}=f^{(0,j_2)}, 0\leq j_2\leq n
\end{equation}
to hold for any $n\le k-1$, we will prove it for $n=k$. In fact, for $1\le j_2\leq k$ the identities
\[
[\cS_k f]^{(0,j_2)}=\big[\cS_{k-j_2}[f_{(0,j_2)}]\big]^{(0,0)},\,[\cS_kf]_{(0,j_2)}=\cS_{k-j_2}[f_{(0,j_2)}]
\]
yield 
\begin{eqnarray*}
[\cS_k^2f]^{(0,j_2)}=[\cS_k(\cS_kf)]^{(0,j_2)}&=&\big[\cS_{k-j_2}[(\cS_kf)_{(0,j_2)}]\big]^{(0,0)}\\
&=&\big[\cS_{k-j_2}(\cS_{k-j_2}f_{(0,j_2)})\big]^{(0,0)}=\big[\cS_{k-j_2}^2f_{(0,j_2)}\big]^{(0,0)}.
\end{eqnarray*}
Since $k-j_2<k$, the induction hypothesis and the duality relations \eqref{dual} imply 
\[
[\cS_k^2f]^{(0,j_2)}=\big[\cS_{k-j_2}^2f_{(0,j_2)}\big]^{(0,0)}=[f_{(0,j_2)}]^{(0,0)}=f^{(0,j_2)}.
\]
What is left is to show \eqref{invo2} to hold for $\bj=(0,0)$. Indeed, from the preceding section, we know that
\begin{equation}\label{a}
[\cS_k^2f]^{(0,0)}(t)=\dfrac{1}{\pi i}\int_{\Gamma}\dfrac{\sum_{n=0}^{k}(\overline{t-\zeta})^{n}[\cS_kf]^{(0,n)}(\zeta)/{n!}}{\zeta-t}d\zeta.
\end{equation}
After  substituting $[\cS_kf]^{(0,n)}(\zeta)$ into \eqref{a} we obtain
\begin{eqnarray*}
[\cS_k^2f]^{(0,0)}(t)&=&-\dfrac{1}{\pi^2}\iint\dfrac{\sum_{n=0}^{k}(\overline{\zeta-\tau})^{n} f^{(0,n)}(\tau)/{n!}}{(\tau-\zeta)(\zeta-t)}\dd\\
&&-\dfrac{1}{\pi^2}\sum_{n=1}^{k}\iint\dfrac{\sum_{p=n}^{k}\binom{p}{n}(\overline{t-\zeta})^{n}(\overline{\zeta-\tau})^{p-n} f^{(0,p)}(\tau)/{p!}}{(\tau-\zeta)(\zeta-t)}\dd.
\end{eqnarray*}
On account of the base induction step we have
\[
-\dfrac{1}{\pi^2}\iint\dfrac{f^{(0,0)}(\tau)}{(\tau-\zeta)(\zeta-t)}\dd=f^{(0,0)}(t). 
\]
Thus, the proof is completed by showing that
\begin{eqnarray*}
-\sum_{n=1}^{k}\dfrac{1}{\pi^2}\iint\dfrac{(\overline{\zeta-\tau})^n f^{(0,n)}(\tau)/{n!}}{(\tau-\zeta)(\zeta-t)}\dd-\\\dfrac{1}{\pi^2}\sum_{n=1}^{k}\iint\dfrac{\sum_{p=1}^{n}\binom{n}{p}(\overline{t-\zeta})^{p}(\overline{\zeta-\tau})^{n-p} f^{(0,n)}(\tau)/{n!}}{(\tau-\zeta)(\zeta-t)}\dd=0,
\end{eqnarray*}
but is clear after using the identity 
\[
(\overline{\zeta-\tau})^n +\sum_{p=1}^{n}\binom{n}{p}(\overline{t-\zeta})^{p}(\overline{\zeta-\tau})^{n-p}=\sum_{p=0}^{n}\binom{n}{p}(\overline{t-\zeta})^{p}(\overline{\zeta-\tau})^{n-p}=(\overline{t-\tau})^n
\]
together with Lemma \ref{lema1}. 
\end{proof}
\section{Hardy decomposition}
From what has already been proved, it follows that the operators $\cP_k^+=\frac{1}{2}(I+\cS_k)$ and $\cP_k^-=\frac{1}{2}(I-\cS_k)$
are projections on $\Li(k+\alpha,\Gamma)$, namely
\[
\cP_k^+\cP_k^+=\cP_k^+,\,\cP_k^-\cP_k^-=\cP_k^-,\,\cP_k^+\cP_k^-=0,\,\cP_k^-\cP_k^+=0.
\]
Consequently, 
\[
\Li(k+\alpha,\Gamma)=\Li^+(k+\alpha,\Gamma)\oplus\Li^-(k+\alpha,\Gamma),
\]
where 
\[
\Li^+(k+\alpha,\Gamma):=\mbox{im}\cP_k^+, \Li^-(k+\alpha,\Gamma):=\mbox{im}\cP_k^-.
\]
The remainder of this section will be devoted to characterize $\Li^\pm(k+\alpha,\Gamma)$. 
\noindent
We recall (see for instance \cite{Ba}) that polyanalytic functions of order $k+1$ (briefly, $(k+1)$-analytic) in some open domain $\Omega\subset\R^2$  are the solutions of the iterated complex equation $\poz^{k+1}F=0$. Each $(k+1)$-analytic function in $\Omega$ can be uniquely represented in the form
\begin{equation}\label{compo}
F(z)=\sum_{m=0}^{k}F_m(z)\bar{z}^m,
\end{equation}
where all the $F_m$'s (called the analytic components of $F$) are analytic functions in the same domain. The expression \eqref{compo} can also be regarded as a definition of a polyanalytic function.
\\
\noindent
The following lemma provides a formula for the $F_m$'s in terms of $F$. Surprisingly, we are not aware of any other work that treats these calculations.
\begin{lemma}\label{Fm}
Le be $F$ a $(k+1)$-analytic function represented by \eqref{compo}. Then we have
\[
F_m(z)=\sum_{j=m}^k\frac{(-1)^{j-m}}{(j-m)!m!}\poz^j F(z)\bar{z}^{j-m},\,\,m=0,\dots, k.
\]
\end{lemma}  
\begin{proof}
We proceed by complete induction on $k$. For $k=1$ we have $F(z)=F_0(z)+F_1(z)\bar{z}$, which is clear from \eqref{compo}. A trivial verification shows that 
$F_1(z)=\poz F(z)$. Hence $F_0(z)=F(z)-\poz F(z)\bar{z}$ and we are done.
\\
\noindent
To conclude from $k$ to $k+1$ it is convenient to start from the representation
\[
F(z)=\sum_{m=0}^{k+1}F_m(z)\bar{z}^m,
\]
or equivalently 
\begin{equation*}
F(z)-F_{k+1}(z)\bar{z}^{k+1}=\sum_{m=0}^{k}F_m(z)\bar{z}^m.
\end{equation*}
Following the assumption for $k$ and the fact that $F_{k+1}(z)=\frac{\poz^{k+1}F(z)}{(k+1)!}$ we have
\[
F_m(z)=\sum_{j=m}^k\frac{(-1)^{j-m}}{(j-m)!m!}\poz^j[F(z)-\frac{\poz^{k+1}F(z)\bar{z}^{k+1}}{(k+1)!}]\bar{z}^{j-m},
\] 
which after some simple algebraic transformations becomes
\begin{equation}\label{alg}
F_m(z)=\sum_{j=m}^k\frac{(-1)^{j-m}}{(j-m)!m!}\poz^j F(z)\bar{z}^{j-m}-\frac{\poz^{k+1}F(z)\bar{z}^{k+1-m}}{m!}\sum_{j=m}^k\frac{(-1)^{j-m}}{(j-m)!(k+1-j)!}.
\end{equation} 
If we apply the classical combinatorial identity
\[
\sum_{j=m}^{k+1}\binom{k+1}{j}\binom{j}{m}(-1)^{j-m}=0,
\]
we get
\[
\sum_{j=m}^{k+1}\frac{(-1)^{j-m}}{(j-m)!(k+1-j)!}=\frac{m!}{(k+1)!}\sum_{j=m}^{k+1}\binom{k+1}{j}\binom{j}{m}(-1)^{j-m}=0
\]
and hence
\begin{equation}\label{comb}
\sum_{j=m}^{k}\frac{(-1)^{j-m}}{(j-m)!(k+1-j)!}=-\frac{(-1)^{k+1-m}}{(k+1-m)!}.
\end{equation}
When \eqref{comb} is substituted in \eqref{alg}, our assertion is obtained.
\end{proof}

\noindent
If $F\in \bC^{k+1}(\Omega)\cap \bC^k(\Omega\cup\Gamma)$ is $(k+1)$-analytic in $\Omega_+$, then it admits the Cauchy type representation (see for instance \cite[p. 173]{Ba},\cite{Be1})
\begin{equation}\label{CF}
 F(z)=\sum_{m=0}^{k}\frac{1}{2\pi i}\int_{\Gamma}\frac{(\overline{z-\zeta})^{m}}{m!(\zeta - z)}\partial^{m}_{\overline{\zeta}}F(\zeta)d\zeta,
\end{equation}
for $z\in\Omega_+$.
\\
\noindent
For the purpose of this paper we also need looking for a representation formula of $(k+1)$-analytic functions in $\Omega_-$. To do this, we use the binomial expansion of $(\overline{z-\zeta})^{m}$ together with Lemma \ref{Fm} to obtain:
\begin{equation}\label{CFnew}
 \sum_{m=0}^{k}\frac{1}{2\pi i}\int_{\Gamma}\frac{(\overline{z-\zeta})^{m}}{m!(\zeta - z)}\partial^{m}_{\overline{\zeta}}F(\zeta)d\zeta=\sum_{m=0}^{k}\bigg[\frac{1}{2\pi i}\int_{\Gamma}\frac{F_m(\zeta)}{\zeta-z}d\zeta\bigg]\bar{z}^m.
\end{equation}
\\
\noindent
The desired representation in $\Omega_-$ is established by our next theorem.
\begin{theorem}
Let $F\in \bC^{k+1}(\Omega_-)\cap \bC^k(\Omega_-\cup\Gamma)$ be $(k+1)$-analytic in $\Omega_-$ with bounded analytic components at $\infty$. Then for $z\in\Omega_-$ we have
\begin{equation}\label{fuera}
F(z)=-\sum_{m=0}^{k}\frac{1}{2\pi i}\int_{\Gamma}\frac{(\overline{z-\zeta})^{m}}{m!(\zeta - z)}\partial^{m}_{\overline{\zeta}}F(\zeta)d\zeta+\sum_{m=0}^{k}F_m(\infty)\bar{z}^m.
\end{equation}
\end{theorem}
\begin{proof}

\noindent
We make use of the Cauchy formula for each analytic bounded component in $\Omega_-$:
\[
\frac{1}{2\pi i}\int_{\Gamma}\frac{F_m(\zeta)}{\zeta-z}d\zeta=-F_m(z)+F_m(\infty).
\]
\noindent
On multiplying through by $\bar{z}^m$ and adding the $m$-terms show that
\[
\sum_{m=0}^{k}\bigg[\frac{1}{2\pi i}\int_{\Gamma}\frac{F_m(\zeta)}{\zeta-z}d\zeta\bigg]\bar{z}^m=-F(z)+\sum_{m=0}^{k}F_m(\infty)\bar{z}^m.
\]

\noindent
Then formula \eqref{fuera} follows directly from \eqref{CFnew}.
\end{proof}
\noindent
In particular, if $F_m(\infty)=0$ for all $m=0,\dots, k$, then formula \eqref{fuera} reads
\begin{equation}\label{fuera0}
F(z)=-\sum_{m=0}^{k}\frac{1}{2\pi i}\int_{\Gamma}\frac{(\overline{z-\zeta})^{m}}{m!(\zeta - z)}\partial^{m}_{\overline{\zeta}}F(\zeta)d\zeta.
\end{equation}
We are now in a position to characterize the spaces $\Li^\pm(k+\alpha,\Gamma)$. 
\begin{theorem} 
The Whitney data $f\in\emph{\Li}(k+\alpha,\Gamma)$ belongs to $\emph{\Li}^+(k+\alpha,\Gamma)$ if and only if there exists a $(k+1)$-analytic function $F$ in $\Omega_+$ which continuously extends to $\Gamma$ together with all its complex partial derivatives $\partial_{\overline{z}}^jF$ for $j\le k$ and such that
\begin{equation}\label{tracecondition}
f^{(0,0)}=F|_\Gamma,\,\,\,f^{(0,j)}=\partial_{\overline{z}}^jF|_\Gamma,\,\,1\le j\le k.
\end{equation}
\end{theorem}
\begin{proof}
By definition, if $f\in\Li^+(k+\alpha,\Gamma)$ there exists $g\in\Li(k+\alpha,\Gamma)$ such that $f=\frac{1}{2}(g+\cS_kg)$. Explicitly, 
\begin{equation}\label{trace}
f^{(\bj)}(t)=\frac{1}{2}\bigg(g^{(\bj)}(t)+\dfrac{\bj!}{\pi i}\int_{\Gamma}\dfrac{\sum_{p=j_2}^{k}\binom{p}{j_2}(\overline{t-\zeta})^{p-j_2}{g^{(0,p)}(\zeta)}/{p!}}{(\zeta-t)^{j_1+1}}d\zeta\bigg), |\bj|\leq k.
\end{equation}
Let us introduce the function
\[
F(z)=\dfrac{1}{2\pi i}\sum_{p=0}^{k}\int_{\Gamma}\frac{(\overline{z-\zeta})^{p}g^{(0,p)}(\zeta)}{p!(\zeta-z)}d\zeta,\,\,z\in\Omega_+.
\]
Of course, $F$ is $(k+1)$-analytic in $\Omega_+$. On the other hand, the smoothness of the kernels $\frac{(\overline{z-\zeta})^{p}}{p!(\zeta-z)}$ for $p\ge 1$ together with the classical Plemelj-Sokhotski formulas imply that $F(z)$ has a continuous trace on $\Gamma$ given by
\[
F(t)=\lim_{\Omega_+\ni z\to t}F(z)=\frac{1}{2}\bigg(g^{(0,0)}(t)+\dfrac{1}{\pi i}\sum_{p=0}^{k}\int_{\Gamma}\frac{(\overline{t-\zeta})^{p}g^{(0,p)}(\zeta)}{p!(\zeta-t)}d\zeta\bigg),\,\,t\in\Gamma,
\]
which, by setting $(\bj)=(0)$ in \eqref{trace}, gives $F(t)=f^{(0,0)}(t)$ in $\Gamma$. \\
In the same manner we can see that
\[
\partial_{\overline{z}}^j F(z)=\dfrac{1}{2\pi i}\sum_{p=j}^{k}\int_{\Gamma}\frac{(\overline{z-\zeta})^{p-j}g^{(0,p)}(\zeta)}{(p-j)!(\zeta-z)}d\zeta
\]
is a polyanalytic function (of order $k+1-j$) in $\Omega_+$ having a continuous trace on $\Gamma$ given by
\[
\partial_{\overline{z}}^j F(t)=\frac{1}{2}\bigg(g^{(0,j)}(t)+\dfrac{1}{\pi i}\sum_{p=j}^{k}\int_{\Gamma}\frac{(\overline{t-\zeta})^{p-j}g^{(0,p)}(\zeta)}{(p-j)!(\zeta-t)}d\zeta\bigg)=f^{(0,j)}(t),\,\,t\in\Gamma,
\]
which proves the necessity. Now suppose that such a polyanalytic function $F$ exists and satisfies \eqref{tracecondition}. Then, the Cauchy integral formula \eqref{CF} gives
\[
F(z)=\dfrac{1}{2\pi i}\sum_{p=0}^{k}\int_{\Gamma}\frac{(\overline{z-\zeta})^{p}f^{(0,p)}(\zeta)}{p!(\zeta-z)}d\zeta,\,\,z\in\Omega_+
\]
and hence
\begin{equation}\label{CFj}
\partial_{\overline{z}}^{j} F(z)=\dfrac{1}{2\pi i}\sum_{p=j}^{k}\int_{\Gamma}\frac{(\overline{z-\zeta})^{p-j}f^{(0,p)}(\zeta)}{(p-j)!(\zeta-z)}d\zeta,\,\,z\in\Omega_+.
\end{equation}
The proof is completed by showing that $f^{(0,j)}=[\cP_k^+(f)]^{(0,j)}$ for all $0\le j\le k$. Indeed, in that case, $f=\cP_k^+(f)\in\Li^+(k+\alpha,\Gamma)$ as claimed.
\\
We have
\begin{eqnarray*}
[\cP_k^+(f)]^{(0,j)}(t)=\frac{1}{2}\bigg(f^{(0,j)}(t)+[\cS_k(f)]^{(0,j)}(t)\bigg)=\\\frac{1}{2}\bigg(f^{(0,j)}(t)+\dfrac{1}{\pi i}\sum_{p=j}^{k}\int_{\Gamma}\frac{(\overline{t-\zeta})^{p-j}f^{(0,p)}(\zeta)}{(p-j)!(\zeta-t)}d\zeta\bigg)\\
=\lim_{\Omega_+\ni z\to t}\dfrac{1}{2\pi i}\sum_{p=j}^{k}\int_{\Gamma}\frac{(\overline{z-\zeta})^{p-j}f^{(0,p)}(\zeta)}{(p-j)!(\zeta-z)}d\zeta.
\end{eqnarray*}
Therefore $[\cP_k^+(f)]^{(0,j)}(t)=f^{(0,j)}(t)$ by \eqref{tracecondition} and \eqref{CFj}. 
\end{proof}
\noindent
The proof of the following result is similar and will be omitted.
\begin{theorem} 
The Whitney data $f\in\emph{\Li}(k+\alpha,\Gamma)$ belongs to $\emph{\Li}^-(k+\alpha,\Gamma)$ if and only if there exists a $(k+1)$-analytic function $F$ in $\Omega_-$ which continuously extends to $\Gamma$ together with all its complex partial derivatives $\partial_{\overline{z}}^jF$ for $j\le k$, with $F_j(\infty)=0,\,j=0,\dots,k$ and
\[
f^{(0,0)}=F|_\Gamma,\,\,\,f^{(0,j)}=\partial_{\overline{z}}^jF|_\Gamma,\,\,1\le j\le k.
\]
\end{theorem}
\noindent
Finally, it should be noted that from what has already been proved, given a Whitney data $f\in\Li(k+\alpha,\Gamma)$, the function 
\[
F(z)=\dfrac{1}{2\pi i}\sum_{p=0}^{k}\int_{\Gamma}\frac{(\overline{z-\zeta})^{p}f^{(0,p)}(\zeta)}{p!(\zeta-z)}d\zeta,\,\,z\in\Omega_+.
\]
is the unique solution of the boundary value problem
\[
\begin{cases}
\poz^{k+1}F(z)=0,\,&z\in\Omega_+\cup\Omega_-,\\
{[\poz^j F]}^+(t)-{[\poz^j F]}^-(t)=f^{(0,j)}(t),\,&t\in\Gamma,\,j=0,\dots, k,\\
F_j(\infty)=0,\,&j=0,\dots, k.
\end{cases}
\]
\section*{Acknowledgments}
The authors gratefully acknowledge the many helpful suggestions and remarks of Prof. Antonio Galbis Verd\'u during the preparation of the manuscript.

\end{document}